 \documentclass[11pt,twoside,reqno,a4paper,makeidx]{amsart}
 \usepackage{amsmath}
 \usepackage{amsthm}
 \usepackage{amsopn}
 \usepackage{a4wide}
 \usepackage{color}
 \usepackage{times}
 \usepackage[numbers,sort&compress]{natbib}
 \usepackage{ifpdf}

 \ifpdf
 \usepackage[pagebackref, colorlinks, hyperindex, linkcolor=red,
 citecolor=blue, urlcolor=blue, hypertexnames=true,
 pdftitle={Factoring Non-negative Operator Valued Trigonometric
   Polynomials in Two Variables}, pdfauthor={Michael A.
   Dritschel}]{hyperref} \fi

\usepackage{enumitem}
\usepackage[mathscr]{eucal}
\usepackage[utopia]{mathdesign}
\newtheorem*{theorem*}{Theorem}
\newtheorem{theorem}{Theorem}[section]
\newtheorem{lemma}[theorem]{Lemma}

\theoremstyle{definition}

\newtheorem{definition}[theorem]{Definition}

\numberwithin{equation}{section}
\setlist{label={$($\roman{enumi}\kern1pt$)$}}
\usepackage[english]{babel}

\usepackage{graphicx}

\newcommand{\fr}{Fej\'er-Riesz }
\newcommand{\ip}[2]{\left<#1,#2\right>}

\newcommand{\ran}{{\mathrm{ran}}\,}
\newcommand{\clran}{\overline{\mathrm{ran}}\,}

\newcommand{\eE}{{\mathfrak E}}
\newcommand{\eF}{{\mathfrak F}}
\newcommand{\eG}{{\mathfrak G}}
\newcommand{\eH}{{\mathfrak H}}
\newcommand{\eK}{{\mathfrak K}}
\newcommand{\eL}{{\mathfrak L}}

\newcommand{\eM}{{\mathfrak M}}
\newcommand{\eW}{{\mathfrak W}}

\newcommand{\etG}{{\tilde{\mathfrak G}}}
\newcommand{\etH}{{\tilde{\mathfrak H}}}
\newcommand{\etK}{{\tilde{\mathfrak K}}}

\renewcommand{\lg}{\eL(\eG)}
\newcommand{\lh}{\eL(\eH)}
\newcommand{\lk}{\eL(\eK)}

\newcommand{\ltk}{\eL(\etK)}
\newcommand{\lgg}{\eL(\eG_1,\eG_2)}
\newcommand{\lhk}{\eL(\eH,\eK)}

\begin{document}

\title[Factoring Two Variable Trigonometric Polynomials] {Factoring
  Non-negative Operator Valued Trigonometric Polynomials in Two
  Variables}

\author{Michael~A.~Dritschel} 

\address{School of Mathematics, Statistics and Physics, Herschel
  Building, University of Newcastle, Newcastle upon Tyne NE1 7RU, UK}

\email{michael.dritschel@ncl.ac.uk}



\thanks{I am grateful to Jim Rovnyak, Hugo Woerdeman, and Scott
  McCullough for their valuable feedback on earlier drafts of this
  paper.}

\date{\today}

\begin{abstract}
  It is shown using Schur complement techniques that on finite
  dimensional Hilbert spaces, a non-negative operator valued
  trigonometric polynomial in two variables with degree $(d_1,d_2)$
  can be written as a finite sum of hermitian squares of at most
  $2d_2$ analytic polynomials.
\end{abstract}

\maketitle

\section{Introduction}
\label{sec:intro}
\bigskip

The \fr theorem on the factorization of a non-negative trigonometric
polynomial in one variable as the hermitian square of an analytic
polynomial is now over $100$ years old.  It has become an essential
tool in both pure mathematics and engineering, especially in signal
processing.  There have been numerous generalizations, and an
especially keen interest in finding a two variable analogue.  This is
provided here, where it is proved that operator valued non-negative
trigonometric polynomials in two variables can be factored as a finite
sum of hermitian squares of analytic polynomials, with tight control
over the number and degrees of the polynomials in the factorization.

\emph{Trigonometric polynomials} are Laurent polynomials in commuting
variables $z=(z_1,\dots,z_r)$ on the $r$-torus $\mathbb T^r$
($\mathbb T$ the unit circle in the complex plane).  This paper is
concerned with polynomials that take their values in the bounded
operators $\lh$ on a Hilbert space $\eH$.  If $n=(n_1,\dots,n_r)$ is a
$r$-tuple of integers, the shorthand $z^n$ is used for
$z_1^{n_1}\cdots z_r^{n_r}$, and $-n$ for $(-n_1,\dots,-n_r)$.  Write
$z_j^*$ for $z_j^{-1}$.  Then any trigonometric polynomial has the
form $p(z) = \sum_n a_n z^n$, where the sum has only finitely many
non-zero terms, and
$p^*(z) := \sum_n a_n^* z^{*\,n} = \sum_n a_n^* z^{-n}$.  Denote by
$\mathrm{deg}\, p = d = (d_1,\dots ,d_r)$ the \emph{degree} of a
trigonometric polynomial $p$; that is, $d_j$ is the largest value of
$|n_j|$ for which $a_n \neq 0$.  A trigonometric polynomial where all
terms have powers $n$ lying in the positive orthant (that is, each
$n_j \geq 0$) is called an \emph{analytic polynomial}.

The $\lh$ valued trigonometric polynomials have a natural operator
system structure.  Those polynomials which satisfy $p^* = p$ are
termed \emph{hermitian}.  The positive cone $\mathcal C$ consists of
those hermitian polynomials $p$ for which $p(z) \geq 0$ for all
$z \in \mathbb T^r$; that is, such that $p(z)$ is a positive operator
for all $z$.  Such $p$ are termed \emph{positive}, and for all $z$ and
$f\in \eH$, $\ip{p(z)f}{f} \geq 0$.  If instead for some
$\epsilon > 0$, $\ip{p(z)f}{f} > \epsilon \|f\|^2$ for all $z$ and all
$f\neq 0$ (equivalently, $p(z) \geq \epsilon 1$), $p$ is said to be
\emph{strictly positive}.  If $p = q^*q$, $q$ an analytic polynomial,
then $p$ is called an \emph{hermitian square}.  Finite sums of such
squares are positive, and it is natural to wonder if this describes
all elements of $\mathcal C$.  For polynomials in one variable, the
\fr theorem gives a positive answer.

Here it is shown that when considering operator valued polynomials in
two variables over a finite dimensional Hilbert space, once again
there is a description of an element in the positive cone as a finite
sum of squares of analytic polynomials, both extending and improving
known results for strictly positive polynomials.  Observe that
requiring polynomials $q_j$ to be analytic when $p = \sum q_j^*q_j$ is
not particularly restrictive, since $z^{N_j} q_j$ will be an analytic
polynomial for sufficiently large $N_j$ and
$p = \sum q_j^* z^{N_j\,*} z^{N_j} q_j$.  This sort of cancellation is
the source of many of the challenges in this context.

For $\lh$ valued trigonometric polynomials in $r$ variables, the cone
$\mathcal C$ is archimedean with the polynomial $1$, which has the
value of the identity $1\in \lh$ for all $z$, as the order unit ---
that is, for any hermitian $p$, there is a positive constant $\alpha$
such that $\alpha 1 \pm p \in \mathcal C$.  There is the obvious
generalization to $M_n(\mathbb C)\otimes \lh$ valued polynomials.  The
cone is in general not closed, so attention is usually restricted to
the set $\mathcal P_N$ of hermitian polynomials of degree less than or
equal to some fixed $N = (n_1,\dots, n_d)$ with norm closed positive
cone $\mathcal C_N$.  However by the cancellation noted above,
hermitian squares can have reduced degree, so it might not be possible
to factor an element of $\mathcal C_N$ using analytic polynomials from
$\mathcal P_N$.

For polynomials taking values in $\mathbb R$, such factorization
problems are central to real algebraic geometry.  It is always
possible to express a scalar valued hermitian trigonometric polynomial
in terms of real polynomials,
\begin{equation*}
  p = \sum_n a_n z^n = \sum_n (\mathrm{Re}\, a_n) (\mathrm{Re}\,
  (x+iy)^n),  \qquad x_j = \mathrm{Re}\, z_j,\ y_j = \mathrm{Im}\,
  z_j.
\end{equation*}
Since $p$ is hermitian and $\mathrm{Re}\, (x+iy)^n$ is a real
polynomial in $x$ and $y$, it follows that $p$ is a real polynomial in
$2r$ variables.  The $r$-torus $\mathbb T^r$ is a set consisting of
those points in $\mathbb R^{2r}$ satisfying the $2r$ constraints,
${\{\pm(1-(x_j^2 + y_j^2)) \geq 0\}}_{j=1}^r$.  These describe a compact
\emph{semialgebraic set}; that is, a compact set given in terms of a
finite collection of polynomial inequalities.  A fundamental problem
then is: Given a semialgebraic set such as $\mathbb T^r$, succinctly
characterize the elements of the positive cone over this set
(generally in terms of ``sums of squares'').  Such a description is
termed a \emph{Positivstellensatz}.  See for example \cite{MR1829790}
and~\cite{MR2383959}.

Recall that in one variable, the Hardy space of functions over the
open unit disk $\mathbb D$ with values in a Hilbert space $\eH$, is
$H^2_{\eH}(\mathbb D) = \{ \sum_{n=0}^\infty c_n z^n : \sum_n
\|c_n\|^2 < \infty \}$.  An $\lh$ valued function $f$ over $\mathbb D$
is called a \emph{multiplier} if for all $g\in H^2_{\eH}(\mathbb D)$,
$f\cdot g \in H^2_{\eH}(\mathbb D)$, where the product is taken
pointwise.  A standard result is that the space of multipliers of
$H^2_{\eH}(\mathbb D)$ equals $H^\infty_{\lh}(\mathbb D)$, the bounded
$\lh$ valued analytic functions on the disk.  Obviously, this space
contains the analytic polynomials.  A multiplier $f$ is said to be
\emph{outer} if the closure of $f\cdot H^2_{\eH}(\mathbb D)$ equals
$H^2_{\eL}(\mathbb D)$, where $\eL$ is a closed subspace of $\eH$.  In
the scalar setting, for a polynomial $f$ this is equivalent to none of
the zeros lying in $\mathbb D$.

As originally formulated, the \fr theorem concerns the factorization
of positive scalar valued trigonometric polynomials in one complex
variable~\cite{Fejer1916}.  It was later proved by
Rosenblum~\cite{Rosenblum1968} (see also~\cite{RR1971,RRbook}) that
for any Hilbert space $\eH$ the theorem remains true for $\lh$ valued
trigonometric polynomials, again in one variable.

\medskip
\begin{theorem}[\fr Theorem]
  \label{thm:fr-thm}
  A positive $\lh$ valued trigonometric polynomial $p$ in a single
  variable $z$ of degree $d$ can be factored as $p=q^*q$, where $q$ is
  an outer $\lh$ valued polynomial of degree $d$.
\end{theorem}

That is, in one variable, the cone $\mathcal C_d$ in $\mathcal P_d$
equals the set of hermitian squares of (outer) analytic polynomials in
$\mathcal P_d$.

There is also a weaker multivariable version of the \fr theorem for
strictly positive trigonometric polynomials~\cite{DW2005} (see
also~\cite{MR2854907}).

\medskip
\begin{theorem}[Multivariable \fr Theorem]
  \label{thm:mv-fr-thm}
  A strictly positive $\lh$ valued trigonometric polynomial $p$ of
  degree $d$ in $r$ variables $z = (z_1,\dots, z_r)$, can be factored
  as a finite sum $p= \sum q_i^*q_i$, where each $q_i$ is an analytic
  $\lh$ valued polynomial.  If $p \geq \epsilon 1$, $\epsilon > 0$,
  then the number of polynomials in the sum and their degrees can be
  bounded in terms of $\epsilon$.
\end{theorem}

Unsurprising, because of the cancellations mentioned above the degree
bounds go to infinity as $\epsilon$ goes to $0$.  Examination of the
proofs yield no obvious choice of closed cone containing all strictly
positive polynomials of a fixed degree, even in two variables.

Suppose $\mathcal R = \{r_j\}_{j = 1}^n$ is a set of real polynomials
on $\mathbb R^r$ and
$\mathcal S = \{x\in \mathbb R^r : r_j(x) \geq 0 \text{ for all }r_j
\in \mathcal R\}$ is a semialgebraic set.  Let $\Delta$ be the set of
functions from $\{1,\dots,m\}$ to $\{0,1\}$.  On $\mathcal S$, there
are various collections of polynomials which are non-negative.
Besides the (finite) sums of squares of polynomials $\sum_i q_i^2$,
which are non-negative everywhere, one also has (again with finite
sums)
\begin{itemize}
\item[$\bullet$] the \emph{quadratic module} : $\sum_j r_j \sum_i
  q_{ji}^2$; and\vskip 2pt
\item[$\bullet$] the \emph{preordering}: $\sum_{\delta\in\Delta} \prod_j
  r_j^{\delta(j)}\sum_i q_{\delta i}^2$.
\end{itemize}

In the scalar setting, the multivariable \fr theorem is a special case
of Schm\"udgen's theorem~\cite{Schmudgen} (see also~\cite{MR759099}).

\begin{theorem}[Schm\"udgen's theorem]
  \label{thm:Schmuedgens-thm}
  Let $\mathcal S$ be a compact semialgebraic set in $\mathbb R^r$
  described by a finite set $\mathcal R$ of polynomials.  Any
  polynomial strictly positive over $\mathcal S$ is in the
  preordering.
\end{theorem}

There are variations and refinements of this result.  For example, if
for each coordinate $x_j$, $\mathcal R$ contains the function
$c_j - x_j^2$ for some $c_j > 0$, then Putinar proved that strictly
positive polynomials are in the quadratic module over
$\mathcal R$~\cite{MR1254128}.  In addition, if in Schm\"udgen's
theorem the set $\mathcal R$ contains at most two polynomials, the
preordering can be replaced by the quadratic module~\cite{MR1829790}.
Cimpri\v{c}~\cite{MR2775777} and Hol and Scherer~\cite{MR2218126} have
extended some of these results to matrix valued polynomials.

Matters become more complicated when positivity replaces strict
positivity.  Both Schm\"udgen's and Putinar's theorems are known not
to hold then.  Indeed, Scheiderer has shown that if the dimension of a
compact semialgebraic set is $3$ or more, there will always be
positive polynomials which are not in the
preordering~\cite{MR2223624}.  On the other hand, he also proved that
under mild restrictions, for a compact two dimensional semialgebraic
set, all positive scalar valued polynomials are in the
preordering~\cite{MR2223624}.  The latter implies in particular that
in two variables, any real valued positive trigonometric polynomial is
a sum of squares of analytic polynomials.

The results of Schm\"udgen, Putinar, and Scheiderer (as well as the
various generalizations mentioned) are intimately tied to the study of
real fields, with no known analogue when the polynomial coefficients
are allowed to come from a ring of operators.  Even the matrix valued
results in~\cite{MR2775777} and~\cite{MR2218126} are likewise
connected, since they are proved by reducing to the scalar valued case
and applying the known theorems.  For this reason, it is particularly
striking that most proofs of the operator \fr theorems take a purely
analytic approach to the problem, though they say nothing about
polynomials vanishing on $\mathbb T^r$ (that is, there is no
\emph{Nullstellensatz}).

This motivates the following generalization of a part of Scheiderer's
work, a factorization theorem with tight bounds for positive operator
valued trigonometric polynomials in two variables.

\begin{theorem*}[Two variable \fr theorem]
  Let $\eK$ be a Hilbert space with $\dim \eK < \infty$.  Given a
  positive $\lk$ valued trigonometric polynomial $Q$ in two variables
  of degree $(d_1,d_2)$, there exist integers
  $m \geq d_B$ such that a positive $\lk$ valued
  trigonometric polynomial $Q$ can be factored as a sum of at most
  $2d_2$ hermitian squares of $\lk$ valued, analytic polynomials with
  degrees bounded by $({d'}_1,2d_2-1)$, ${d'}_1 = \max\{d_A, m\}$.
\end{theorem*}

The case where either $d_1$ or $d_2$ equals $0$ can be excluded as it
is covered by the single variable \fr theorem.  The principle
techniques are a hybrid of Rosenblum's approach to the operator \fr
theorems using the Lowdenslager criterion and an elaboration of the
Schur complement approach found in \cite{MAD2004} and \cite{DW2005}.
Though these methods are now well known, there are several new twists.
As usual, a trigonometric polynomial is associated with a Toeplitz
operator, though in the two variable setting, this becomes a Toeplitz
operator of Toeplitz operators.  Positive trigonometric polynomials
correspond to finite degree positive Toeplitz operators (that is,
having only finitely many nonzero diagonals).  These can be viewed as
either being ``bi-infinite'' in each variable (indexed from $-\infty$
to $+\infty$) or ``singly infinite''(indexed from $0$ to $+\infty$) ,
or even some mixture of these.  In the singly infinite case, the
multiplication operators identified to the variables are commuting
isometries (in fact, unilateral shifts), while in the bi-infinite case
they are commuting unitaries (bilateral shifts).  The interplay
between these turns out to be important here.

Collecting entries of a selfadjoint finite degree Toeplitz operator
into large enough blocks, one gets a tridiagonal Toeplitz operator.
Calling the main diagonal entry $A$ and the off-diagonal entries $B$
and $B^*$, this Toeplitz operator is positive if and only if there is
a positive operator $M$ such that $\begin{pmatrix} A-M & B^* \\ B & M
\end{pmatrix} \geq 0$, and in this case the set of all such $M$ forms
a norm closed compact, convex set $\mathcal M$~\cite{MAD2004}.  The
elements of $\mathcal M$ play a similar role to the Schur complement
supported on the $(1,1)$ entry of the Toeplitz operator, which happens
to be the largest element of $\mathcal M$.  The Schur complement gives
rise to the outer factorization (unique up to multiplication by a
diagonal unitary) of the positive Toeplitz matrix.  The smallest
element of $\mathcal M$ gives the ``co-outer'' factorization.  Of
particular interest will be those extremal polynomials for which
$\mathcal M$ is a singleton.

In two variables, a complication arises in that the elements of
$\mathcal M$ may not consist solely of Toeplitz operators, and in
particular, the largest element will in general not be Toeplitz.
Nevertheless, there will always be a closed convex subset
$\mathcal M_T$ of $\mathcal M$ consisting of Toeplitz operators.  The
extremal case when $\mathcal M_T$ (rather than $\mathcal M$) is a
singleton is of central importance.  It is proved that $A$ can always
be replaced by a Toeplitz $\hat A \leq A$ so that
$\mathcal M_T = \{\hat M\}$ is a singleton.  However, while $\hat A$
and $\hat M$ will be associated to bounded trigonometric functions, it
is not evident that they necessarily have finite degree, and so
correspond to polynomials.

Another difficulty then arises.  While the \fr theorem guarantees the
existence of a factorization for a positive single variable
trigonometric \emph{polynomial}, it is well known that not all
positive trigonometric \emph{functions} can be factored as a hermitian
square of an analytic function.  Some restriction, as in for example
Szeg\H{o}'s theorem~\cite{RRbook}, is usually needed.  Despite this,
it will follow by minimality that $\hat{A}-\hat{M}$ and $\hat{M}$ have
outer factorizations $E^*E$ and $F^*F$.  It will also imply that
$B = F^*GE$, where $G = V^*_FV_E$ is Toeplitz with a bi-infinite
unitary extension, and $V_E$ and $V_F$ inner (so analytic and
isometric).  This, with the degree bounds for $B$, then lead to degree
bounds for $E$ and $F$ when $\eH$ is finite dimensional, and so for
$\hat{A}$ and $\hat{M}$.  An argument using the one variable \fr
theorem then finishes the proof of the two variable theorem.

Why is it not possible to use the same ideas to factor positive
trigonometric polynomials in three or more variables?  While the
theorem presented is about trigonometric polynomials in two variables,
frequent use is made of results concerning polynomials in one
variable, and outer factorizations of these, especially in showing
that the degrees of $\hat A$ and $\hat M$ mentioned above are finite.
Outerness (interpreted appropriately) does not necessarily apply to
analytic factors in two variables.  Indeed, there are examples of
positive polynomials in three or more variables which cannot be
factored as a sum of squares of polynomials, indirectly indicating
that there will exist positive polynomials in two variables without
outer factorizations.  Outer factorizations of multivariable
trigonometric polynomials are explored further in~\cite{DW2005}.

\section{Toeplitz and analytic operators, and their relation to
  trigonometric polynomials}
\label{sec:toepl-analyt-oper}

A \emph{shift operator} is an isometry $S$ on a Hilbert space $\eH$
with trivial unitary component in its Wold decomposition.  It is then
natural to write for some Hilbert space $\eG$,
$\eH = \eG \oplus \eG \oplus \cdots$ (identified with the Hardy space
$H^2(\eG)$), and
\begin{equation*}
S (h_0,h_1,\dots)^t = (0,h_0,h_1,\dots) ^t 
\end{equation*}
when the elements of $\eH$ are written as column vectors (here ``$t$''
indicates transpose).  The \emph{multiplicity} of $S$ is then
$\dim \eG = \dim\ker S^*$.

Fix a shift $S$.  If $T,A\in\lh$, say that $T$ is \emph{Toeplitz} if
$S^*TS=T$, and that $A$ is \emph{analytic} if $AS=SA$.  To distinguish
this case from that of Toeplitz operators on $L^2$ spaces introduced
below, such operators will be referred to as being \emph{singly
  infinite} Toeplitz operators.  Viewed as an operator on $H^2(\eG)$,
pre-multiplication of $T$ by $S^*$ has the effect of deleting the
first row of $T$ and shifting $T$ upwards by one row, while
pre-multiplication by $S$ shifts $T$ downwards by a row and setting
the first row entries to $0$.  Likewise, post-multiplication by $S$
deletes the first column and shifts left by one column.  Hence as
matrices with entries in $\lg$, such Toeplitz and analytic operators
have the forms
  \begin{equation}\label{sep30a}
  T = \begin{pmatrix}
    T_0 & T_{-1} & T_{-2} & \cdots\\
    T_1 & T_0    & T_{-1} & \ddots \\
    T_2 & T_1  & T_0 & \ddots \\
    \vdots    & \ddots  & \ddots & \ddots
  \end{pmatrix} ,
\qquad
 A = \begin{pmatrix}
    A_0 & 0 & 0 & \cdots\\
    A_1 & A_0    & 0 & \ddots \\
    A_2 & A_1  & A_0 & \ddots \\
    \vdots    & \ddots  & \ddots & \ddots
  \end{pmatrix} .
\end{equation}

An analytic operator $A$ is termed \emph{outer} if $\clran A$ is a
subspace of $\eH$ of the form $H^2(\eF)$ for some closed subspace
$\eF = \clran A_0$ of $\eG$; equivalently,
$\clran A = \bigoplus_0^\infty \clran A_0$ reduces $S$.

So far it has been assumed that the entries $T_j$ of the Toeplitz
operators $T$ are in $\lg$, and while this is necessary if $T\geq 0$,
it is also natural to more generally consider Toeplitz and analytic
operators where $T_j \in \lgg$, $\eG_1$, $\eG_2$, Hilbert spaces.

Now consider Laurent and analytic polynomials
$Q(z) = \sum_{k=-d_-}^{d_+} Q_k z^k$ and $P(z) = \sum_{k=0}^d P_k z^k$
on $\mathbb T$ with coefficients in $\lg$.  Refer to
$\deg_-(Q) = d_-$, $\deg_+(Q) = d_+$, and $\deg(Q) = d_- + d_+$ as the
\emph{degrees} of $Q$, and $\deg(P) = d$ as the degree of $P$,
assuming that $Q_{d_+}$, $Q_{-d_-}$, and $P_d$ are nonzero while
$Q_{d_+ + \ell}$, $Q_{-d_+ - \ell}$, and $P_{d+\ell}$ are zero for
$\ell \geq 1$.  The formulas
  \begin{equation}\label{aug31a}
  T_Q = \begin{pmatrix}
    Q_0 & Q_{-1} & Q_{-2} & \cdots\\
    Q_1 & Q_0    & Q_{-1} & \ddots \\
    Q_2 & Q_1  & Q_0 & \ddots \\
    \vdots    & \ddots  & \ddots & \ddots
  \end{pmatrix} ,
\qquad
 T_P = \begin{pmatrix}
    P_0 & 0 & 0 & \cdots\\
    P_1 & P_0    & 0 & \ddots \\
    P_2 & P_1  & P_0 & \ddots \\
    \vdots    & \ddots  & \ddots & \ddots
  \end{pmatrix} 
\end{equation}
then define bounded linear operators on $\eH$.  The operator $T_Q$ is
Toeplitz, while $T_P$ is analytic, and $d_-$, $d_+$, and $d$ are
likewise called the \emph{degrees} of $T_Q$ and $T_P$.  Even if $Q$
and $P$ are not polynomials, but are nevertheless bounded functions,
the operators $T_Q$ and $T_P$ will be bounded.  Moreover,
\begin{enumerate}
\item[$\bullet$] $Q(z) \ge 0$ for all $z\in\mathbb T$ if and only if
  $T_Q \ge 0$; \smallskip

\item[$\bullet$] $Q(z) = P(z)^*P(z)$ for all $z \in\mathbb T$ if and
  only if $T_Q = T_P^*T_P$.
\end{enumerate}

Recall that $Q(z) \geq 0$ means that for all $g\in \eG$,
$\ip{Q(z)g}{g} \geq 0$, while $Q(z) > 0$ if for some $\epsilon > 0$,
$\ip{Q(z)g}{g} \geq \epsilon \|g\|^2$.  Write $ Q \geq 0$ and $Q > 0$
for $Q(z) \ge 0$, respectively $Q(z) > \epsilon 1$, for all
$z\in\mathbb T$.

An analytic function $P(z)$ is \emph{outer} if the analytic Toeplitz
operator $T_P$ is outer.  The \fr theorem (Theorem~\ref{thm:fr-thm})
can be restated in terms of Toeplitz operators: \textsl{A positive
  Toeplitz operator $T \in \eL(H^2(\eG))$ of finite degree
  $d = d_+ = d_-$ has the form $F^*F$, where $F\in \eL(H^2(\eG))$ is
  an outer analytic operator of the same degree $d$ as $T$.}

While the \fr theorem states that any bounded positive Toeplitz
operator of finite degree has an analytic, and hence outer
factorization, the same need not be true if the degree is not finite.
Some additional condition needs to be imposed (see \cite{RRbook},
especially Section~3.4, and \cite[Lemma~2.3]{MR2743422}).  The next
result is a mild generalization of Lowdenslager's criterion, which
ensures an outer (and so analytic) factorization.  First, some
notation.  If $T \in \eL(H^2(\eG))$ is positive and Toeplitz (so
$T = S^*TS$, $S$ the unilateral shift) and $T = F^*F$ is a
factorization, then $FS = S_FF$, where $S_F$ is referred to as the
\emph{Lowdenslager isometry} associated to $F$.  Here, unlike in
\cite{RRbook}, the term ``factorization'' does not imply analyticity
of $F$.

\begin{lemma}[Lowdenslager criterion]
  \label{lem:Lowdenslager-criterion}
  Let $T \in \eL(H^2(\eG))$ be positive and Toeplitz.  The following
  are equivalent:
  \begin{enumerate}
  \item[(i)] There is a factorization $T = F^*F$ where $F$ is outer;
  \item[(ii)] For some factorization $T = G^*G$, the Lowdenslager
    isometry $S_G$ is a unilateral shift;
  \item[(iii)] For every factorization $T = H^*H$, the Lowdenslager
    isometry $S_H$ is a unilateral shift.
  \end{enumerate}
  In this case, the multiplicities of all the Lowdenslager isometries
  are equal.
\end{lemma}

The proof that $($\textit{i}$)$ implies $($\textit{iii}$)$ and
$($\textit{ii}$)$ implies $($\textit{i}$)$ is essentially the same as
for the classical Lowdenslager criterion~\cite{RRbook}.

A consequence of the Lowdenslager criterion is that if $E$ is analytic
then it has an \emph{inner-outer factorization}; that is, $E = VF$,
where $F$ is outer and $V$ is \emph{inner} (so analytic and
isometric).

One can extend singly infinite Toeplitz operators to
$\etH = \dots \oplus \eG\oplus \eG \oplus \eG \oplus \dots = L^2(\eG)$
simply by continuing each of the diagonals.  The shift operator becomes
the unitary bilateral shift.  The resulting \emph{bi-infinite}
Toeplitz operator is positive if and only if the same is true for the
corresponding singly infinite Toeplitz operator.

All other notions considered so far carry over naturally to the
multi-index / multivariable setting.  Only the two index / variable
case is examined, the version for three or more then being evident.
Suppose that $S_2$ is a shift operator on $\eH = \bigoplus_{j_2 =
  0}^\infty \eG$, and that $S_1$ is a shift operator on $\eG =
\bigoplus_{j_1 = 0}^\infty \eK$.  If $T$ is a Toeplitz operator on
$\eH$ with the property that each $T_{j_2}$ is a Toeplitz operator on
$\eG$, say that $T$ is a \emph{bi-Toeplitz} operator (or
\emph{multi-Toeplitz} more generally).  Call $T$ \emph{bi-analytic}
(respectively, \emph{multi-analytic}) if $T$ is analytic and each
$T_j$ is analytic.  It is sometimes convenient to shift back and forth
to the bi-infinite Toeplitz setting in one of the variables.  If
$T$ is a bi-Toeplitz operator on $\bigoplus_{j_1 = 0}^\infty
\bigoplus_{j_2 = 0}^\infty \eK$, the entries are naturally labeled by
two indices, $(j_1,j_2)$.  Let $\etG = \bigoplus_{j_2 = 0}^\infty \eK$
and $\etH = \bigoplus_{j_1 = 0}^\infty \bigoplus_{j_2 = 0}^\infty \eK
= \bigoplus_{j_1 = 0}^\infty \etG$.  The indices of $T$ can be
interchanged to get another operator $\tilde T$ on $\etH$.  The
exchange is implemented via a permutation of rows and columns
corresponding to conjugation with the unitary operator $W : \etH \to
\eH$ having the identity $1$ in the entries labeled with
$((j_1,j_2),(j_2,j_1))$ and $0$ elsewhere.

As in the single variable setting, there are Laurent and analytic
polynomials with coefficients in $\lk$, but now in $z = (z_1,z_2)$,
where $z_1$ and $z_2$ commute.  These look like
\begin{equation*}
  Q(z) = \sum_{k_2=-m_2}^{m_2} \left(\sum_{k_1=-m_1}^{m_1}
    Q_{k_2,k_1} z_1^{k_1}\right) z_2^{k_2} \quad\text{and}\quad
  P(z) = \sum_{k_2=0}^{m_2} \left(\sum_{k_1=0}^{m_1} P_{k_2,k_1}
    z_1^{k_1}\right) z_2^{k_2}.
\end{equation*}
Set $Q_{j_1,j_2} =0$ whenever $j_1\notin [-m_1,m_1]$ or $m_2\notin
[-m_2,m_2]$, and set $P_{j_1,j_2}=0$ whenever $j_1\notin [0,m_1]$ or
$j_2\notin [0,m_2]$.  This results in trigonometric polynomials in the
variable $z_2$ with coefficients which are trigonometric polynomials
in the variable $z_1$.  Much as before, the formulas
  \begin{equation}\label{aug31b}
  T_Q = (Q_{j_2-k_2,j_1-k_1})_{(j_2,j_1), (k_2,k_1) \in \mathbb
    N\times \mathbb N}
\qquad
 T_P = (Q_{j_2-k_2,j_1-k_1})_{(j_2,j_1), (k_2,k_1) \in \mathbb
    N\times \mathbb N}
\end{equation}
define bounded operators on $\eH$, the first being bi-Toeplitz and the
second bi-analytic.  If indices are interchanged and ${\tilde T}_Q$
and ${\tilde T}_P$ are viewed as operators on $\etH$, this amounts to
taking $Q$ and $P$ as polynomials in $z_1$ with coefficients which are
polynomials in $z_2$.  The pairs $(m_{1,\pm},m_{2,\pm})$ are referred
to as the \emph{degrees} of the $Q$ (equivalently, degrees of $T_Q$),
if the coefficients of the form $Q_{\pm m_{1,\pm},\pm m_{2,\pm}}$ is
nonzero, while $Q_{j_1,j_2} =0$ if $j_1\notin [-m_{1,-},m_{1,+}]$ or
$j_2\notin [-m_{2,-},m_{2,+}]$, When $Q$ is positive,
$m_{1+} = m_{1-}$ and $m_{2+} = m_{2-}$, so these degrees are
unambiguously written as $(m_1,m_2)$.

In analogy with the one variable case,
\begin{enumerate}
\item[$\bullet$] $Q(z) \ge 0$ for all $z\in\mathbb T^2$ if and only
  if $T_Q \ge 0$; \smallskip

\item[$\bullet$] $Q(z) = \sum_j P_j(z)^*P_j(z)$ for all $z
  \in\mathbb T^2$ if and only if $T_Q = \sum_j T_{P_j}^*T_{P_j}$.
\end{enumerate}

\section{Schur complements}
\label{sec:Schur-complements}

Schur complements play an essential role in several proofs of the
operator \fr theorem~\cite{MAD2004,DW2005}.  A survey of their use in
this way can be found in~\cite{MR2743422}.  Here is the definition.

\begin{definition}\label{nov4e}
  Let $\eH$ be a Hilbert space and $0 \leq T\in\lh$.  Let $\eK$ be a
  closed subspace of $\eH$, and $P_\eK \in\lhk$ the orthogonal
  projection of $\eH$ onto $\eK$.  Then there is a unique operator $0
  \leq M = M(T,\eK)\in\lk$ called the \emph{Schur complement of $T$
    supported on $\eK$}, such that
  \begin{enumerate}
  \item[(i)] $T-P_\eK^* M P_\eK \ge 0$;
  \item[(ii)] if $\widetilde M \in\lk$, $\widetilde M \ge 0$, and
    $T-P_\eK^* \widetilde M P_\eK \ge 0$, then $\widetilde M \le M$.
  \end{enumerate}
\end{definition}
 
There are several equivalent means of obtaining the Schur complement.
For example, if $T = \begin{pmatrix} A & B^* \\ B & C \end{pmatrix}$
on $\eK \oplus \eK^\bot$, $M$ is found by
\begin{equation}
\label{eq:2}
  \ip{Mf}{f} = \inf_{g\in\eK^\bot} \ip{
    \begin{pmatrix} A & B^* \\ B & C \end{pmatrix}
    \begin{pmatrix} f \\ g \end{pmatrix}}{
    \begin{pmatrix} f \\ g \end{pmatrix}}, \qquad f\in \eK.
\end{equation}

Suppose that $T$ is a positive Toeplitz operator of finite degree
$d = d_+ = d_-$ on $\eH= \bigoplus_0^\infty \eK$.  By grouping the
entries into $d\times d$ sub-matrices, $T$ can be taken to be
tridiagonal.  Write $\etK$ for $\bigoplus_0^{d-1} \eK$ and view
$\eH= \bigoplus_0^\infty \etK$.  Then
\begin{equation}
\label{eq:3}
  T =
  \begin{pmatrix}
    A & B^* & 0 & \cdots \\
    B & A & B^* & \ddots \\
    0 & B & A & \ddots \\
    \vdots & \ddots & \ddots & \ddots
  \end{pmatrix}.
\end{equation}
Let $M_+$ denote the Schur complement supported on the first copy of
$\etK$.  Then by \eqref{eq:2} and the Toeplitz structure of $T$ (see
\cite{MAD2004} or \cite{MR2743422}),
\begin{equation}
  \label{eq:4}
  \begin{pmatrix}
    A-M_+ & B^* \\
    B & M_+
  \end{pmatrix}
  \geq 0.
\end{equation}
In fact, $T\geq 0$ if and only if there is some $M\geq 0$ such that
the inequality \eqref{eq:4} holds with $M$ in place of $M_+$.  In this
case, write $\mathcal M$ for the set of positive operators $M$
satisfying \eqref{eq:4}.  This set is norm closed and convex with
maximal element equal to the Schur complement $M_+$.  There is also a
minimal element $M_-$ which is constructed by finding the maximal
element $N_+$ such that
\begin{equation*}
  \begin{pmatrix}
    N_+ & B^* \\
    B & A-N_+
  \end{pmatrix}
  \geq 0
\end{equation*}
and setting $M_- = A - N_+$.  Evidently, $N_+$ is the Schur complement
supported on the first copy of $\etK$ in
\begin{equation*}
  \begin{pmatrix}
    A & B & 0 & \cdots \\
    B^* & A & B & \ddots \\
    0 & B^* & A & \ddots \\
    \vdots & \ddots & \ddots & \ddots
  \end{pmatrix}.
\end{equation*}

In the context of positive Toeplitz operators, Schur complements have
a certain inheritance property, in that if $T_Q$ is Toeplitz on
$H^2(\mathcal H)$ and $M_n(T_Q)$ is the Schur complement supported on
the upper left $n\times n$ corner of $T_Q$, then $M_n(M_{n+1}(T_Q)) =
M_n(T_Q)$; that is, the Schur complement on the upper left $n\times n$
corner of the Schur complement on the upper left $(n+1)\times (n+1)$
corner of $T_Q$ is the same as the Schur complement on the upper left
$n\times n$ corner of $T_Q$.  In addition, if $\deg T_Q \leq n$, then
\begin{equation}
  \label{eq:5}
  M_{n+1}(T_Q) =
  \begin{pmatrix}
    Q_0 & Q_1^* & \cdots & Q_n^* \\
    Q_1 & & & \\
    \vdots & & M_n(T_Q) & \\
    Q_n & & &
  \end{pmatrix},
\end{equation}
where some $Q_j = 0$ if $j > n$.  This enables the construction of the
\fr factorization in the one variable case, via
\begin{equation*}
  \begin{split}
    & R_n(T) = M_{n+1}(T_Q) -
    \begin{pmatrix}
      & & & 0 \\
      & M_n(T_Q) & & \vdots \\
      & & & 0 \\
      0 & \cdots & 0 & 0
    \end{pmatrix}
    \\[3pt] = &
    \begin{pmatrix}
      P_d^* & 0 & \cdots & \cdots & 0 \\
      \ddots & \ddots & \ddots & \ddots & \vdots \\
      P_0^* & \ddots & \ddots & \ddots & \vdots \\
      0 & \ddots & \ddots & \ddots & 0 \\
      \vdots & \ddots & \ddots & \ddots & P_d^* \\
      \vdots & \ddots & \ddots & \ddots & \vdots \\
      0 & \cdots & \cdots & 0 & P_0^*
    \end{pmatrix}
    \begin{pmatrix}
      P_d & \ddots & P_0 & 0 & \cdots & \cdots & 0 \\
      0 & \ddots & \ddots & \ddots & \ddots & \ddots & \vdots \\
      \vdots & \ddots & \ddots & \ddots & \ddots & \ddots & \vdots \\
      \vdots & \ddots & \ddots & \ddots & \ddots & \ddots & 0 \\
      0 & \ddots & \ddots & 0 & P_d & \ddots & P_0
    \end{pmatrix}.
  \end{split}
\end{equation*}
See~\cite{DW2005} or~\cite{MR2743422} for more details.

A particularly interesting situation, termed \emph{extremal}, occurs
when $\mathcal M = \{M\}$, a singleton.  In this case, if there are
factorizations $A - M = E^*E$ and $M = F^*F$, then $B = F^*UE$, where
$U$ is unitary from $\clran E$ to $\clran F$.  However such a
factorization of $B$ with a unitary for some element $M\in \mathcal M$
does not necessarily guarantee extremality.  In order to examine this
more carefully, the following test is introduced.

\begin{lemma}
  \label{lem:extremal-test}
  Suppose that
  \begin{equation*}
    \mathcal M = \left\{ M :
        \begin{pmatrix}
          A - M & B^* \\ B & M
        \end{pmatrix}
        \geq 0\right\}
  \end{equation*}
  is non-empty with maximal element $M_+$, minimal element $M_-$, and
  $M_0 = \tfrac{1}{2}(M_+ + M_-)$.  Let $A-M_\pm = E_\pm^*E_\pm$ and
  $A-M_0 = E_0^*E_0$, $M_\pm = F_\pm^*F_\pm$, $M_0 = F_0^*F_0$, and
  $B = F_\pm^*G_\pm E_\pm = F_0^*G_0 E_0$, where
  $G_\pm : \clran E_\pm \to\clran F_\pm$ and
  $G_0: \clran E_0 \to\clran F_0$ are contractions.  The set
  $\mathcal M$ is a singleton if and only if the operators $G_+$,
  $G_-$ and $G_0$ are unitary.
\end{lemma}

\begin{proof}
  Let $M_+$, $M_-$ be the maximal and minimal elements of $\mathcal
  M$.  Suppose that $G_+$, $G_-$ and $G_0$ are the corresponding
  contractions as in the statement of the lemma.  It is
  straightforward to verify that $G_+$ is an isometry and $G_-$ is a
  co-isometry.  Hence if $\mathcal M$ is a singleton, $G_+ = G_- =
  G_0$ is unitary.

  Conversely, assume that for every $M\in \{ M_+,M_-,M_0\}$, there are
  factorizations as in the statement of the lemma, where the operators
  $G_+$, $G_-$ and $G_0$ are unitary.  Without loss of generality, by
  absorbing $G_\pm$ into $E_\pm$ or $F_\pm$, it is possible to take
  $G_\pm = 1$ on $\clran E_\pm = \clran F_\pm$.

  Since $M_- \leq M_+$, there exist contractions
  $H_-: \clran E_- \to \clran E_+$, $H_+: \clran F_+ \to \clran F_-$
  with dense ranges such that $E_+ = H_- E_-$ and $F_- = H_+ F_+$.
  Let $H_- = V_- |H_-|$ be the polar decomposition.  Then
  $E_+^*E_+ = E_+^*V_-V_-^*E_+$, so replacing $E_+$ by $V_-^*E_+$ if
  necessary, $\clran E_+ = \clran E_-$ and $H_-$ is a positive
  contraction.  Similarly, one can take $\clran F_+ = \clran F_-$ and
  $H_+$ a positive contraction.  So
  $\clran E_+ = \clran E_- = \clran F_+ = \clran F_-$ and
  \begin{equation*}
    \begin{split}
      B &= F_-^* E_- = F_+^* H_+ E_- \\
      & = F_+^* E_+ = F_+^* H_- E_-.
    \end{split}
  \end{equation*}
  Consequently, $H_+ = H_-$.  Denote this operator by $H$.

  The set $\mathcal M$ is convex, so
  \begin{equation*}
    \begin{split}
      0 & \leq
      \begin{pmatrix}
        A - \tfrac{1}{2}(M_+ + M_-) & B^* \\
        B & \tfrac{1}{2}(M_+ + M_-)
      \end{pmatrix}
      = \frac{1}{2}\left[
        \begin{pmatrix}
          E_+^*E_+ & B^* \\
          B & F_+^*F_+
      \end{pmatrix}
      +
        \begin{pmatrix}
          E_-^*E_- & B^* \\
          B & F_-^*F_-
      \end{pmatrix}
      \right] \\
      & = \frac{1}{2}\left[
        \begin{pmatrix}
          E_-^*H^2E_- & E_-^*HF_+ \\
          F_+^*HE_- & F_+^*F_+
        \end{pmatrix}
        +
        \begin{pmatrix}
          E_-^*E_- & E_-^*HF_+ \\
          F_+^*HE_- & F_+^*H^2F_+
        \end{pmatrix}
      \right] \\
      & = \frac{1}{2}
      \begin{pmatrix}
        E_-^* & 0 \\ 0 & F_+^*
      \end{pmatrix}
      \begin{pmatrix}
          1+H^2 & 2H \\
          2H & 1+H^2
        \end{pmatrix}
      \begin{pmatrix}
        E_- & 0 \\ 0 & F_+
      \end{pmatrix}.
    \end{split}
  \end{equation*}

  Recall that $M_0 = \tfrac{1}{2}(M_+ + M_-)$, $A-M_0 = E_0^*E_0$,
  $M_0 = F_0^*F_0$, and $B = F_0^*G_0E_0$, where by assumption,
  $G_0:\clran E_0 \to\clran F_0$ is unitary.  Then the Schur
  complement supported on the top left corner of $\begin{pmatrix} A -
    M_0 & B^* \\ B & M_0 \end{pmatrix}$ is zero.  Thus the Schur
  complement of the top left corner of $\begin{pmatrix} 1+H^2 & 2H \\
    2H & 1+H^2 \end{pmatrix}$ must also be zero.  Since $1+H^2$ is
  invertible, it is a standard fact that this Schur complement equals
  $(1+H^2) - 4H(1+H^2)^{-1}H$, and so $1 - 2H^2 + H^4 = 0$.  Since $H
  \geq 0$, this implies that the only point in the spectrum of $H$ is
  $\{1\}$, and so $H = 1$.  From this it follows that $M_+ = M_-$;
  that is, $\mathcal M$ is a singleton.
\end{proof}

The next theorem shows that even if $\mathcal M$ is not a singleton,
it is possible to replace $A$ by $\hat A \leq A$ so that it is.

\begin{theorem}
  \label{thm:extremals-exist}
  Suppose that the set
  \begin{equation*}
    \mathcal M = \left\{ M :
        \begin{pmatrix}
          A - M & B^* \\ B & M
        \end{pmatrix}
        \geq 0\right\}
  \end{equation*}
  is not empty.  Then there exists $\hat A \leq A$ such that the set
  \begin{equation*}
    \hat{\mathcal M} = \left\{ M :
        \begin{pmatrix}
          \hat A - M & B^* \\ B & M
        \end{pmatrix}
        \geq 0\right\}
  \end{equation*}
  is a singleton.
\end{theorem}

\begin{proof}
  Assume that the set $\mathcal M$ is not a singleton.  Then
  $\mathcal M$ has distinct maximal and minimal elements $M_+$ and
  $M_-$.  Set $M_0 = \tfrac{1}{2}(M_+ + M_-)$.  By
  Lemma~\ref{lem:extremal-test}, for $M_*$ one of these three and
  $A-M_* = E^*E$, $M_* = F^*F$, $B = F^*GE$, $G: \clran E \to\clran F$
  is a non-unitary contraction.

  The Schur complement supported on the top left corner of
  $\begin{pmatrix} A-M_* & B^* \\ B & M \end{pmatrix}$ is
  $D_+ = E^*(1-G^*G)E$, while that on the lower right corner is
  $D_- = F^*(1-GG^*)F$.  So if $G$ is not isometric, then
  $D_+ \neq 0$, while if it is not coisometric, then $D_- \neq 0$.

  In the first case, set $M_1 = M_*$, and $A_1 = A - D_+ \geq 0$.
  Then
  \begin{equation*}
    \begin{pmatrix}
      A_1 - M_1 & B^* \\ B & M_1
    \end{pmatrix}
    = 
    \begin{pmatrix}
      (A-M_*) - D_+ & B^* \\ B & M_*
    \end{pmatrix}
    \geq 0.
  \end{equation*}
  Likewise, if $G$ is not co-isometric, set $M_1 = M_* - D_-$ and
  $A_1 = A - D_- = (A-M_*) - (M_*-D_-) \geq 0$.  Then
  \begin{equation*}
    \begin{pmatrix}
      A_1 - M_1 & B^* \\ B & M_1
    \end{pmatrix}
    = 
    \begin{pmatrix}
      A-M_* & B^* \\ B & M_*-D_-
    \end{pmatrix}
    \geq 0.
  \end{equation*}
  In either case, $A_1 \lneq A$ and
  \begin{equation*}
    \mathcal M_1 :=
    \left\{ M :
      \begin{pmatrix}
        A_1 - M & B^* \\ B & M
      \end{pmatrix}
      \geq 0\right\} \subseteq \mathcal M.
  \end{equation*}

  Let $\mathcal A$ be the set of all $A' \leq A$ for which
  $\begin{pmatrix} A'-M & B^* \\ B & M \end{pmatrix} \geq 0$ for some
  $M \geq 0$.  If $\mathcal C \subset \mathcal A$ is a decreasing
  chain, then the elements of $\mathcal C$ converge strongly to
  $A_0 \geq 0$.  The corresponding maximal choice of $M$ for the
  elements of the chain themselves form a decreasing chain, and so
  they too converge strongly.  Hence there exists $M$ such that
  $\begin{pmatrix} A_0 - M & B^* \\ B & M \end{pmatrix} \geq 0$; that
  is, $A_0 \in \mathcal A$.  An application of Zorn's lemma gives a
  minimal $\hat{A} \in \mathcal A$.  The set of $M$ such that
  $\begin{pmatrix} \hat{A}-M & B^* \\ B & M \end{pmatrix} \geq 0$ is a
  singleton, since otherwise the construction given above yields
  $A' \lneq \hat{A}$.
\end{proof}

Now consider the two index / variable case.  Let
$Q(z) = \sum_{j_2 = -d_2}^{d_2} {\hat Q}_{j_2}(z_1) z_2^{j_2}$ be an
$\lk$ valued trigonometric polynomial (since the positive case will be
of interest, the simplified form for the maximal and minimal indices
is used).  As in the last section, there is an associated Toeplitz
operator
$T_{Q_1}(z_1) := ({\hat Q}_{j_2 - k_2})_{j_2, k_2 = 0}^\infty$, the
entries of which are $\lk$ valued trigonometric polynomials
${\hat Q}_{j_2} = \sum Q_{j_1,j_2} z_1^{j_1}$ of degree at most $d_1$.
Following the example of the single variable case, group these into
$d_2 \times d_2$ submatrices.  Then with
$\etK = M_{d_2}(\mathbb C) \otimes \eK$, the entries are $\ltk$ valued
trigonometric polynomials in $z_1$.  This in turn is equivalent to a
tridiagonal Toeplitz operator $T$ as in \eqref{eq:3}, with $A$ and $B$
Toeplitz operators with entries in
$\lg = \eL(\bigoplus_0^\infty \etK)$, where
\begin{equation}
  \label{eq:1}
  \begin{split}
    \deg_0 A &:= \deg_+ A = \deg_- A = \sup\{\deg_\pm {\hat Q}_{j_2} :
    0 \leq j_2 \leq d_2-1\} \\
    \deg_\pm B &= \sup\{\deg_\pm {\hat Q}_{j_2} : 1 \leq j_2 \leq
    d_2\}, \qquad \deg B =\deg_+ B + \deg_- B.
  \end{split}
\end{equation}
Just to emphasize, the bi-Toeplitz operator obtained in this way has
outer level corresponding to the variable $z_2$ and inner level
corresponding to $z_1$.  In other words, the Toeplitz operators which
are the entries of the tridiagonal Toeplitz operator correspond to
functions in the variable $z_1$.

It has been assumed that the operators $A$ and $B$ are singly infinite
Toeplitz operators.  If they are instead replaced with the
corresponding bi-infinite Toeplitz operators (so acting on $\lg =
L^2(\etK)$) --- call them $\tilde A$ and $\tilde B$ and the resulting
tridiagonal operator $\tilde T$ --- then $T$ is positive if and only
if $\tilde T$ is positive.  This is therefore a singly infinite
Toeplitz operator with coefficients which are bi-infinite Toeplitz
operators.  Write $\tilde S$ for the bilateral shift on $\etK$ and
suppose that ${\tilde M}_+$ is the Schur complement appearing in the
resulting version of \eqref{eq:4}.  Since $\tilde A$ and $\tilde B$
are bi-infinite Toeplitz operators, they are invariant under
conjugation with either $\tilde S$ or ${\tilde S}^*$.  Consequently,
\begin{equation}
  \label{eq:6}
  \begin{pmatrix}
    \tilde A - \tilde S{\tilde M}_+{\tilde S}^* & {\tilde B}^* \\
    \tilde B & \tilde S{\tilde M}_+{\tilde S}^*
  \end{pmatrix}
  \geq 0 \qquad \text{and} \qquad
  \begin{pmatrix}
    \tilde A - {\tilde S}^*{\tilde M}_+ \tilde S & {\tilde B}^* \\
    \tilde B & {\tilde S}^*{\tilde M}_+ \tilde S
  \end{pmatrix}
  \geq 0 .
\end{equation}
Hence ${\tilde S}^*{\tilde M}_+ \tilde S \leq {\tilde M}_+$ and
$\tilde S{\tilde M}_+{\tilde S}^* \leq {\tilde M}_+$.  On the
other hand, $\tilde S{\tilde S}^* = 1$, so conjugating both sides of
the first of these inequalities by $\tilde S$ gives
${\tilde M}_+ \leq {\tilde S}^*{\tilde M}_+\tilde S$, and so equality
holds.  Equality holds likewise for the other inequality.  In other
words, the Schur complement in this case is Toeplitz.
The same argument works with the minimal element ${\tilde M}_-$.  In
neither case is it necessary to assume that $\tilde A$ and $\tilde B$
have finite degrees.  There is no immediate guarantee that the degree
of $\tilde M$ is finite, even if the degrees of $\tilde A$ and
$\tilde B$ are.

The discussion is summarized in the next lemma.

\begin{lemma}
  \label{lem:2-inf-T-have-T-Scs}
  Suppose that $\tilde A$ and $\tilde B$ are bounded bi-infinite
  Toeplitz operators with entries in $\ltk$ (they need not be of
  finite degree).  Let
  $\tilde{\mathcal M} = \{\tilde M : \begin{pmatrix} \tilde A-\tilde M
    & {\tilde B}^* \\ \tilde B & \tilde M
  \end{pmatrix} \geq 0 \}$, and assume that $\tilde{\mathcal M}$ is
  non-empty.  Let $\tilde{\mathcal M}_T$ be the (non-empty) subset of
  elements of $\tilde{\mathcal M}$ which are Toeplitz.  Then the
  maximal and minimal elements of $\tilde{\mathcal M}$, ${\tilde M}_+$
  and ${\tilde M}_-$ are in $\tilde{\mathcal M}_T$ as maximal and
  minimal elements.  This set is a closed and convex.
\end{lemma}

There is a refined version of Theorem~\ref{thm:extremals-exist} for
bi-infinite Toeplitz operators.

\begin{theorem}
  \label{thm:doubly-inf-extremal}
  Let $\tilde A$ and $\tilde B$ be bi-infinite Toeplitz operators
  (they need not be of finite degree), and suppose that the set
  \begin{equation*}
    \tilde{\mathcal M} = \left\{ \tilde M :
        \begin{pmatrix}
          \tilde A - \tilde M & {\tilde B}^* \\ \tilde B & \tilde M
        \end{pmatrix}
        \geq 0\right\}
  \end{equation*}
  is not empty.  Then there exists a minimal bi-infinite Toeplitz
  operator $0 \leq \hat A \leq \tilde A$ such that the set
  \begin{equation*}
    \hat{\mathcal M} = \left\{\tilde M :
        \begin{pmatrix}
          \hat A - \tilde M & {\tilde B}^* \\ \tilde B & \tilde M
        \end{pmatrix}
        \geq 0\right\} = \{\hat M\},
  \end{equation*}
  is a singleton.  In this case, $\hat M$ is a bounded bi-infinite
  Toeplitz operator.
\end{theorem}

\begin{proof}
  This is essentially a repeat of the proof of
  Theorem~\ref{thm:extremals-exist}, taking into account
  Lemma~\ref{lem:2-inf-T-have-T-Scs}, which guarantees that the
  operator $M_*$ in the proof of Theorem~\ref{thm:extremals-exist} is
  bi-infinite Toeplitz.  The same argument used for
  Lemma~\ref{lem:2-inf-T-have-T-Scs} implies that the operators
  $D_\pm$ are bi-infinite Toeplitz operators, and so $A_1$ and $M_1$
  are as well.  Strong limits of bounded sequences of bi-infinite
  Toeplitz operators are bi-infinite Toeplitz operators so the chains
  considered there have upper bounds in the class, and the result
  again follows from an application of Zorn's lemma.
\end{proof}

The arguments just used with bi-infinite Toeplitz operators do not
work in the singly infinite setting with the unilateral shift.
Indeed, the Schur complement in this case will generally not be
Toeplitz.  If it is, it can be shown by arguments to follow that it is
possible to factor the bivariate trigonometric polynomial with
analytic polynomials of the same degrees, and there are well known
examples for which this is not possible~\cite{MAD2004}.  Nevertheless,
restricting back to singly infinite operators, the following is
obtained.

\begin{lemma}
  \label{lem:Toeplitz-Ms}
  Suppose that $A$ and $B$ are bounded (singly infinite) Toeplitz
  operators with entries in $\ltk$ (they need not be of finite
  degree).  Let $\mathcal M = \{M : \begin{pmatrix} A-M & B^* \\ B & M
  \end{pmatrix} \geq 0 \}$, and assume that $\mathcal M$ is non-empty.
  Then there exists a closed, convex subset $\mathcal M_T$ of
  $\mathcal M$ with maximal and minimal elements, all the elements of
  which are Toeplitz.  Furthermore, there exists a minimal Toeplitz
  operator $0 \leq \hat A \leq A$ such that the set
  \begin{equation*}
    \hat{\mathcal M}_T = \left\{ \hat M :
        \begin{pmatrix}
          \hat A - \hat M & B^* \\ B & \hat M
        \end{pmatrix}
        \geq 0\right\}
  \end{equation*}
  is a singleton.
\end{lemma}

The elements of $\mathcal M_T$ come from restricting the set
$\tilde{\mathcal M}_T$ of operators $\tilde M$ in the bi-infinite
setting satisfying
\begin{equation*}
  \begin{pmatrix}
    \tilde A - \tilde M & {\tilde B}^* \\
    \tilde B & \tilde M
  \end{pmatrix}
  \geq 0.
\end{equation*}
The set $\mathcal M_T$ will therefore have maximal and minimal
elements since $\tilde{\mathcal M}$ does, and will be a singleton when
$\tilde{\mathcal M}$ is.

In general, positive Toeplitz operators need not have outer
factorizations.  The next theorem provides a useful exception, and is
the key ingredient in the proof of the main result.


\begin{theorem}
  \label{thm:min-A-gives-outer-fact}
  Suppose that $A$ and $B$ are bounded (singly infinite) Toeplitz
  operators with entries in $\ltk$, that $B$ has finite degree
  $(d_+,d_-)$ on $H^2(\eH)$ (that is, the smallest values such that
  $BS^{d_+}$ and $B^* S^{d_-}$ are analytic), and assume that
  $\mathcal M = \{M :
  \begin{pmatrix} A-M & B^* \\ B & M \end{pmatrix} \geq 0 \} \neq
  \emptyset$.  Let $\hat{A}$ be the minimal Toeplitz operator and
  $\hat M$ the corresponding unique Toeplitz operator such that
  $\begin{pmatrix} \hat A - \hat M & B^* \\ B & \hat M \end{pmatrix}
  \geq 0$.  Then there exist outer $E$ and $F$ such that
  $\hat{A}-\hat{M} = E^*E$ and $\hat{M} = F^*F$.  Furthermore, there
  are inner $V_E$, $V_F$ with bi-infinite unitary extensions having
  equal ranges such that for $G = V_F^* V_E$, $B = F^*GE$.
\end{theorem}

\begin{proof}
  The proof proceeds as follows.  First, it is proved that
  $\hat{A}-\hat{M}$ and $\hat{M}$ have outer factorizations.  Then it
  is shown that $R$ itself has an outer factorization, and this
  implies the existence of analytic $R_E$ and $R_F$ such that $R =
  \begin{pmatrix} R_E^* \\ R_F^* \end{pmatrix} \begin{pmatrix} R_E &
    R_F \end{pmatrix}$.  Inner-outer factorizations of $R_E$ and $R_F$
  yield inner operators $V_E$ and $V_F$, and then uniqueness of
  $\hat{M}$ implies that for the bi-infinite extensions of these,
  ${\tilde{V}}_F^*{\tilde{V}}_E$ is unitary.

  Let $\tilde{A}$ and $\tilde{M}$ be the bi-infinite extensions of
  $\hat{A}$ and $\hat{M}$ to $L^2(\eH)$, $\tilde{B}$ that of
  $B$, and $P_0$ the orthogonal projection of $L^2(\eH)$ onto
  $H^2(\eH)$.  For the time being, choose any factorizations
  $\tilde{A} - \tilde{M} = {\tilde{E}}^*\tilde{E}$ and
  $\tilde{M} = {\tilde{F}}^*\tilde{F}$.  By
  Lemma~\ref{lem:extremal-test},
  $\tilde{B} = {\tilde{F}}^*\tilde{G}\tilde{E}$, where
  $\tilde{G}: \clran \tilde{E} \to \clran \tilde{F}$ is unitary.  Set
  $E = \tilde{E}P_0^*$, $F = \tilde{F}P_0^*$, $G =
  P^*_{\clran(\tilde{F}P_0)} \tilde{G} P_{\clran(\tilde{E}P_0)}$.
  Then $\hat{A} - \hat{M} = E^*E$, $\hat{M} = F^*F$, and  $B =
  F^*GE$.

  Since $\hat{A} - \hat{M}$ and $\hat{M}$ are Toeplitz, there exist
  Lowdenslager isometries $V_E$ and $V_F$ such that $V_E E = E S$ and
  $V_F F = F S$ (see Lemma~\ref{lem:Lowdenslager-criterion} and the
  paragraph preceding it).  Let $V_E = S_E \oplus U_E$,
  $V_F = S_F \oplus U_F$ be Wold decompositions, where $S_E$, $S_F$
  are shift operators and $U_E$, $U_F$ are unitary.  The operator
  $B^* S^{d_-}$ is analytic, so for all $j \in \mathbb N$,
  $S^j (E^* G^* V_F^{d_-}) = (E^* G^* V_F^{d_-})V_F^j$.  Hence if
  $Q_j$ is the projection onto $\bigoplus_{i = 0}^j \eH$, then for any
  $j$, $ Q_j (E^* G^* V_F^{d_-})V_F^{j+1} = 0$.  Thus with
  $\mathcal N_F = \ran U_F$ and $j \geq 0$,
  \begin{equation*}
    \{0\} = Q_j S^{j+1} (E^* G^* V_F^{d_-}) \mathcal N_F = Q_j E^* G^*
    V_F^{d+j+1} \mathcal N_F = Q_j E^* G^* \mathcal N_F.
  \end{equation*}
  Hence $E^* G^* \mathcal N_F = \{0\}$.

  Let $P_F$ be the orthogonal projection with range equal to
  $\clran F \ominus \ran U_F$.  Since $S_F P_F F = P_F F S$,
  $F^*P_F F \geq 0$ is Toeplitz.  By the above calculations,
  $B = F^* P_F G E$ and $\|G\| \leq 1$, so
  \begin{equation*}
    \begin{pmatrix} 
      \hat A - F^*P_F F & B^* \\ B & F^*P_F F
    \end{pmatrix}
    \geq
    \begin{pmatrix} 
      \hat A - \hat{M} & B^* \\ B & F^*P_F F
    \end{pmatrix}
    =
    \begin{pmatrix} 
      E^*E & E^* G^* P_F F \\ F^* P_F G E & (P_F F)^* P_F F
    \end{pmatrix}
    \geq 0.
  \end{equation*}
  By uniqueness of $\hat{M}$, $P_F = 1_{\clran F}$ and
  $\mathcal N_F = \{0\}$.  Hence $F$ is analytic from
  $H^2(\eH)$ to
  $H^2(\eH_F) = \bigoplus_{0}^\infty (\ker S_F^*)$.  Since
  $\clran F = H^2(\eH_F)$, $F$ is outer.

  A similar argument with $BS^{d_+}$ shows that for
  $\mathcal N_E = \ran U_E$, $F^* G \mathcal N_E = 0$, and so if $P_E$
  is the orthogonal projection with range equal to
  $\clran E \ominus \ran U_E$, then $E^*P_EE$ is Toeplitz, $B = F^* G
  P_E E$, and
  \begin{equation*}
    \begin{pmatrix} 
      (\hat{A} - E^*(1_E - P_E)E) - \hat{M} & B^* \\ B & \hat{M}
    \end{pmatrix}
    =
    \begin{pmatrix} 
      E^*E - E^*(1_E - P_E)E) & B^* \\ B & \hat{M}
    \end{pmatrix}
    =
    \begin{pmatrix} 
      (P_E E)^* (P_E E) & E^* P_E G^* F \\ F^* G P_E E & F^*F
    \end{pmatrix}
    \geq 0.
  \end{equation*}
  Minimality of $\hat{A}$ then yields that $\mathcal N_E = \{0\}$ and
  $E$ is analytic, and in fact outer, from $H^2(\eH)$ to $\clran E =
  H^2(\eH_E) = \bigoplus_{0}^\infty (\ker S_E^*)$.  By Lowdenslager's
  criterion (Lemma~\ref{lem:Lowdenslager-criterion}), $E$ and $F$ may
  be taken to be outer on $H^2(\eH)$, and this is what is done.  We
  denote the closure of their ranges by $H^2(\eE)$ and $H^2(\eF)$,
  respectively.

  The equalities $F^*GE = B = S^*BS = F^*S^*GS E$ and the fact that
  $E$ and $F$ are outer imply that $G$ is Toeplitz.  Furthermore, if
  $\tilde{S}$, $\tilde{S}_E$, $\tilde{S}_F$ are the bilateral (and
  hence unitary) extension of $S$, $S_E$, and $S_F$, then
  $\tilde{E}\tilde{S} = \tilde{S}_E \tilde{E}$,
  ${\tilde{F}}^*\tilde{S}_F = \tilde{S} {\tilde{F}}^*$, and
  $\tilde{G}\tilde{S}_E = \tilde{S}_F \tilde{G}$.  By construction,
  $G$ is unique, $\tilde{G}: L^2(\eE) \to L^2(\eF)$ is unitary, and
  $\tilde{B} = {\tilde{F}}^* \tilde{G} \tilde{E}$.  Decompose
  $L^2(\eH) = H^2(\eH)^\bot \oplus H^2(\eH))$,
  $L^2(\eE) = H^2(\eE)^\bot \oplus H^2(\eE))$,
  $L^2(\eF) = H^2(\eF)^\bot \oplus H^2(\eF)$.  With respect to these
  decompositions,
  \begin{equation*}
    \tilde{F}^* =
    \begin{pmatrix}
      {F'}^* & Q_F^* \\ 0 & F^*
    \end{pmatrix}, \qquad
    \tilde{G} =
    \begin{pmatrix}
      G' & D \\ D' & G
    \end{pmatrix},
    \qquad\text{and} \qquad
    \tilde{E} =
    \begin{pmatrix}
      E' & 0 \\ Q_E & E
    \end{pmatrix}.
  \end{equation*}
  Matrix multiplication verifies that $\tilde{B}$ is the bi-infinite
  extension of $B = F^*GE$.  Because $F$ is outer,
  $\clran \tilde{F} = L^2(\eF)$ and
  $\clran F' = H^2(\eF)^\bot$.

  Set
  \begin{equation*}
    W :=
    \begin{pmatrix}
      D & 0 \\
      G & 1
    \end{pmatrix}
    : H^2(\eE) \oplus H^2(\eF) \to L^2(\eF).
  \end{equation*}
  The entries of $W$ are Toeplitz and the left column is isometric.
  Also, $\ran W = \ran D \oplus H^2(\eF)$.  Since
  $\tilde{G}\tilde{S}_E = \tilde{S}_F \tilde{G}$, $\tilde{S}_F W = W
  (S_E \oplus S_F)$.  The columns of
  \begin{equation*}
    \begin{pmatrix}
      D \\
      G
    \end{pmatrix}
    =
    \begin{pmatrix}
      \vdots & \ddots & \ddots & \ddots \\
      G_{-2} & G_{-3} & G_{-4} & \ddots \\
      G_{-1} & G_{-2} & G_{-3} & \ddots \\ \hline
      G_0 & G_{-1} & G_{-2} & \ddots \\
      G_1 & G_0 & G_{-1} & \ddots \\
      \vdots & \ddots & \ddots & \ddots \\
    \end{pmatrix}
  \end{equation*}
  are isometric with orthogonal ranges.  Let $\eL$ be the range of the
  first column.  Then
  $\ran \begin{pmatrix} D \\ G \end{pmatrix} = \bigoplus_0^\infty \eL$
  and $\tilde{S}_F$ acts as a unilateral shift on this space, as well
  as on ${0} \oplus H^2(\eF)$.  Set $S_W = \tilde{S}_F | \clran W$.
  Since
  $\tilde{S}_F W = W \begin{pmatrix} S_E & 0 \\ 0 &
    S_F \end{pmatrix}$, $S_W$ maps $\clran W$ isometrially to itself.
  The goal now is to show that this is a unilateral shift.

  Suppose that
  \begin{equation*}
    f=
    \begin{pmatrix}
      f_1 \\ f_2
    \end{pmatrix}
    \in \bigcap_{n=0}^\infty {\tilde{S}_F}^n \clran W.
  \end{equation*}
  Then for all $n$, there exists $h_n = \begin{pmatrix}h_{1n} \\ h_{2n}
  \end{pmatrix} \in H^2(\eE) \oplus H^2(\eF)$ such that
  \begin{equation*}
    f -
    \begin{pmatrix}
      DS_E^n h_{1n} \\ G S_E^n h_{1n} + S_F^n h_{2n}
    \end{pmatrix}
    = f - W
    \begin{pmatrix}
      S_E^n & 0 \\ 0 & S_F^n
    \end{pmatrix}
    h_n
    = f - {\tilde{S}_F}^n 
    \begin{pmatrix}
      D h_{1n} \\ Gh_{1n} + h_{2n}
    \end{pmatrix}
    = f - {\tilde{S}_F}^n W h_n \to 0.
  \end{equation*}
  
  Since ${\tilde{S}_F}^{*\,n}$ has the form
  $\begin{pmatrix} {S'_F}^{*\,n} & Q_{F,n} \\ 0 &
    {S_F}^{*\,n} \end{pmatrix}$, it follows that
  ${S_F}^{*\,n} f_2 \to 0$ and ${S_F}^* G S_E = G$, and so
  $h_{2n} \to -Gh_{1n}$.  Hence
  $h_n - \begin{pmatrix}h_{1n} \\ -Gh_{1n} \end{pmatrix} \to 0$, and
  so
  \begin{equation*}
    f - {\tilde{S}_F}^n W
    \begin{pmatrix}h_{1n} \\ -Gh_{1n} \end{pmatrix}
    = f - W 
    \begin{pmatrix}
      S_E^n & 0 \\ 0 & S_F^n
    \end{pmatrix}
    \begin{pmatrix} h_{1n} \\ {S_F}^{*\,n}GS_E^nh_{1n} \end{pmatrix}
    = f - 
    \begin{pmatrix}
      DS_E^n h_{1n} \\ (1-S_F^n{S_F}^{*\,n}) G S_E^n h_{1n}
    \end{pmatrix}
    \to 0.
  \end{equation*}
  For any $n$, the first $n$ entries of
  $(1-S_F^n{S_F}^{*\,n}) G S_E^n h_{1n}$ are $0$, so the limit as
  $n\to \infty$ gives $f_2 = 0$.

  In the Wold decomposition of $S_W$, the unitary part acts on a
  reducing subspace
  $\mathcal U = \bigcap S_W^n\clran W$, which is a subspace of
  $H^2(\eF)^\bot \oplus \{0\}$ by what ws just shown.  Since
  $S_W\mathcal U = \mathcal U$ and $\clran W$ is invariant for
  $\tilde{S}_F$, it follows that $\mathcal U$ reduces $\tilde{S}_F$.
  Thus for all $n$,
  ${\tilde{S}_F}^n \begin{pmatrix} f_1 \\ 0 \end{pmatrix} \in \mathcal
  U$ has the form $\begin{pmatrix} x_n \\ 0 \end{pmatrix}$.  Therefore
  \begin{equation*}
    \begin{pmatrix}
      f_1 \\ 0
    \end{pmatrix}
    = {\tilde{S}_F}^{*\,n}{\tilde{S}_F}^n
    \begin{pmatrix}
      f_1 \\ 0
    \end{pmatrix}
    =
    \begin{pmatrix}
      \begin{pmatrix}
        \vdots \\ f_{1,n+1} \\ f_{1n} \\ 0 \\ \vdots \\ 0
      \end{pmatrix}
      \\[10pt] 0
    \end{pmatrix},
  \end{equation*}
  and consequently $f_1 = 0$.  So
  $\mathcal U \subseteq \clran W \ominus (\ran W)^\bot = \{0\}$; that
  is, $S_W$ is a unilateral shift and
  $S_W W = W \begin{pmatrix} S_E & 0 \\ 0 & S_F \end{pmatrix}$.  Write
  $H^2(\eW)$ for $\clran W$ under the action of $S_W$.
  
  As a result,
  \begin{equation*}
    H:= W
    \begin{pmatrix}
      E & 0 \\ 0 & F
    \end{pmatrix}
    : H^2(\eH) \oplus H^2(\eH) \to H^2(\eW)
  \end{equation*}
  is analytic since the matrix on the right is outer, so has dense
  range in $H^2(\eE) \oplus H^2(\eF)$, the space on which $W$ is
  defined.  A simple calculation shows that
  \begin{equation*}
    R:=
    \begin{pmatrix}
      \hat A - \hat{M} & B^* \\ B & \hat{M}
    \end{pmatrix}
    = H^*H.
  \end{equation*}
  Decomposing $H = \begin{pmatrix} H_1 & H_2 \end{pmatrix}$,
  $H_1, H_2: H^2(\eH) \to H^2(\eW)$ are analytic with
  $H_1^* H_1 = \hat A - \hat{M} = E^*E$ and
  $H_2^* H_2 = \hat{M} = F^*F$ outer factorizations.  Hence
  $H_1 = V_EE$, $H_2 = V_FF$, with $V_E: H^2(\eE) \to H^2(\eW)$,
  $V_F: H^2(\eF) \to H^2(\eW)$ inner, and so in $B = F^* G E$,
  $G = V_F^*V_E$.

  Let ${\tilde V}_E, {\tilde V}_F$ be the bi-infinite extensions of
  $V_E, V_F$, respectively.  By assumption,
  $\tilde G = {\tilde V}_F^* {\tilde V}_E$ is unitary, so
  $\ran {\tilde V}_E = \ran {\tilde V}_F$.  Set
  $\eM = \overline{\ran V_E \vee \ran V_F}$.  Since $V_E$ and $V_F$
  are analytic, $\tilde{S}_\eH\eM = S_\eH \eM \subseteq \eM$, which
  means that $S_0 := S_\eH | \eM$ is a shift operator.  Setting
  $\eK = \ker S_0^*$, it follows that $\eM = H^2(\eK)$ and
  $S_0 = S_\eK$.  Then
  $L^2(\eK) = \ran \tilde{V}_E = \ran \tilde{V}_F$.  In other words,
  $\tilde{V}_E$ and $\tilde{V}_F$ map unitarily into $L^2(\eK)$, and
  $V_E$ and $V_F$ are analytic and isometric into $H^2(\eK)$.
  Consequently, $V_EE$ and $V_FF$ map analytically into $H^2(\eK)$,
  $A-M = E^*V_E^*V_EE$, $M = F^*V_F^*V_FF$, and $B = F^*V_F^*V_EE$.
\end{proof}

The question of degree bounds in the previous theorem is addressed
next.


\begin{theorem}
  \label{thm:min-A-w-degree-bds}
  Assume $\eH$ is finite dimensional and that the conditions of the
  last theorem hold.  Then the operators $A-M$, $M$, $E$, and $F$ have
  finite degree.
\end{theorem}

\begin{proof}
  The notation of the proof of the last result is maintained.  While
  the space $\eK \subset \ran \tilde{V}_E$, it may not be in
  $\ran V_E$.  However, the Toeplitz structure of $V_E$ compensates
  for this.

  Choose an orthonomal basis $\{f_k\}_{k=1}^d$ for $\eK$, where
  $d < \infty$.  It is clear from the definition that
  $\dim \eK \leq \dim \eH$.  Fix $\frac{1}{\sqrt{d}} > \epsilon > 0$.  Let
  $\{\tilde{g}_k\}$ be an orthonormal subset of $L^2(\eE)$ such that
  $\tilde{V}_E \tilde{g}_k = f_k$.  Set
  $\etG = \bigvee_k \tilde{g}_k$.  Write $P_{\eE}$ for the orthogonal
  projection from $L^2(\eE)$ onto $H^2(\eE)$.  Let $\epsilon > 0$.
  For $m_k$ sufficiently large,
  $\|\tilde{S}_{\eE}^{m_k} \tilde{g}_k - P_{\eE}
  (\tilde{S}_{\eE}^{m_k} \tilde{g}_k) \| < \epsilon$.  Let
  $m = \sup_k m_k$.  Then
  $\|\tilde{S}_{\eE}^m \tilde{g}_k - P_{\eE} (\tilde{S}_{\eE}^m
  \tilde{g}_k) \| < \epsilon$ for all $k$ as well.  Set
  $g_k = P_{\eE} (\tilde{S}_{\eE}^m \tilde{g}_k)$ and
  $\eG = \bigvee_k g_k$.

  Note that the $m$th rows of
  $V_E$ and $\tilde{V}_E$ are of the form
  \begin{equation*}
    \begin{split}
      K_m &= \begin{pmatrix}
        V_m & V_{m-1} & \cdots & V_0 & 0 & \cdots
      \end{pmatrix} \\
      \tilde{K}_m &= \begin{pmatrix}
        \cdots V_{m+2} & V_{m+1} & V_m & V_{m-1} & \cdots & V_0 & 0 & \cdots
      \end{pmatrix}. \\
    \end{split}
  \end{equation*}
  Let $L_m = \begin{pmatrix} \cdots V_{m+2} & V_{m+1} \end{pmatrix}$,
  so that $\tilde{K}_m = \begin{pmatrix} L_m & K_m \end{pmatrix}$.  By
  the above construction, $h_k := K_m g_k \in \eK$, and $\| h_k - f_k
  \| < \epsilon$ for all $k$.

  If $h := \sum_k \alpha_k h_k = 0$, then for
  $f := \sum_k \alpha_k f_k$,
  \begin{equation*}
    \|f\| = \|h-f\| = \left\|\sum_k \alpha_k (h_k - f_k)\right\| \leq
    \epsilon \sum_k |\alpha_k| \leq \epsilon \sqrt{d \sum_k
      |\alpha_k|^2} = \epsilon \sqrt{d} \|f\|.
  \end{equation*}
  Hence $\|f\| = 0$, and thus $\alpha_k = 0$ for all $k$.  It follows
  then that $\bigvee_k h_k = \eK$, meaning that $\ran K_m = \eK$.
  Note that it will also be the case that $\ran K_n = \eK$ for all
  $n \geq m$.

  The operator $BS^{d_+}$ is analytic, so for $R_F = V_F F$,
  $R_F^* S_{\eK}^{d_+}$ is analytic on
  $\clran V_E E = \ran V_E$, which is invariant under $S_{\eK}$ by
  analyticity of $V_E$.  In other words,
  \begin{equation*}
    R_F^* S_{\eK}^{d_+} S_{\eK} | \ran V_E = S_{\eH} R_F^*
    S_{\eK}^{d_+} | \ran V_E.
  \end{equation*}
  In particular then, since $S_{\eK}^m H^2(\eK) \subseteq \ran V_E$,
  \begin{equation*}
    R_F^* S_{\eK}^{d_+ + m} S_{\eK} |_{H^2(\eK)} = S_{\eH} R_F^*
    S_{\eK}^{d_+ + m} |_{H^2(\eK)}.
  \end{equation*}
  Hence $R_F^* S_{\eK}^{d_+ + m}$ is Toeplitz and analytic (so lower
  triangular).  The operator $R_F^*$ is co-analytic (upper
  triangular), hence $R_F$ has degree at most $d_+ + m$, as does $F$
  since it is outer.  Therefore as $\hat{M} = R_F^*R_F$,
  $\deg \hat{M} \leq d_+ + m$.

  An identical argument using analyticity of $B^*S^{d_-}$ shows that
  for some $n \geq 0$, $E^*V_E^*S_{\eK}^{d_- + n}$ is analytic.  As a
  consequence, for $R_E = V_E E$, $\deg (R_E) \leq d_- + n$, as does
  $E$ since it is outer.  Since $\hat{A} - \hat{M} = R_E^* R_E$, its
  degree is also bounded by $d_- + n$.
\end{proof}

\section{A two variable factorization theorem}
\label{sec:2var-fr}

This section contains the proof of the main result.  First some
notation.  Suppose that $T_Q$ is a Toeplitz operator associated to a
trigonometric polynomial $Q$ of bi-degree $(d_1,d_2)$ in the variables
$(z_1,z_2)$, and that $T_Q$ has degree $d_2$ with coefficients
$R_0, \dots, R_{d_2}$ which are themselves Toeplitz of degree bounded
by $d_1$.  Write
$d_A^\pm = \max\{\deg_\pm R_j,\ j = 0,\dots, d_2-1\}$,
$d_A = \max \{d_A^+, d_A^-\}$, and
$d_B^\pm = \max\{\deg_\pm R_j,\ j = 1,\dots, d_2\}$,
$d_B = \max\{d_B^+, d_B^-\}$.

Recall that in one variable, the \fr theorem (\ref{thm:fr-thm}) gives
a factorization of any operator valued trigonometric polynomial with
very strong degree bounds.  However in two variables, even in the
scalar valued case, there are no uniform bounds based on the degree of
the polynomials being factored (see Section~5 of
\cite{scheiderer-deg-bds}).  By considering a direct sum of scalar
polynomials all of the same degree but with factorization degrees
growing, one can construct an operator valued trigonometric polynomial
of finite degree without a factorization in terms of polynomials of
bounded degree.  However, if the coefficients are operators on a
finite dimensional Hilbert space, polynomial factorization is
possible.

\begin{theorem}
  \label{thm:2d-fejer-riesz-thm}
  Let $\eK$ be a Hilbert space with $\dim \eK < \infty$.  Given a
  positive $\lk$ valued trigonometric polynomial $Q$ in two variables,
  there exist integers $m \geq d_B$ such that a positive $\lk$ valued
  trigonometric polynomial $Q$ can be factored as a sum of at most
  $2d_2$ hermitian squares of $\lk$ valued, analytic polynomials with
  degrees bounded by $({d'}_1,2d_2-1)$, ${d'}_1 = \max\{d_A, m\}$.
\end{theorem}

Of course the roles of $z_1$ and $z_2$ can be reversed, potentially
altering the number and degrees of the polynomials in the
factorization.

\begin{proof}[Proof of Theorem~\ref{thm:2d-fejer-riesz-thm}]
  As observed above, the trigonometric polynomial $Q$ is associated to
  a bi-Toeplitz operator $T_Q$ in such a way that $T_Q$ is Toeplitz of
  degree $d_2$ with coefficients $R_0, \dots, R_{d_2}$ which are
  Toeplitz of finite degrees bounded by $d_A$, $d_B^\pm$, and $d_B$ as
  defined above.

  Collect the coefficients into $d_2\times d_2$ blocks to form a
  tridiagonal Toeplitz operator.  The blocks as they stand are not
  Toeplitz.  However, there is a unitary conjugation (call the unitary
  $V$) collecting the $(i,j)$th entries into $d_2\times d_2$ blocks
  acting on a space $\etK = \bigoplus_1^{d_2} \eK$.  This results in a
  positive tridiagonal Toeplitz operator
  \begin{equation*}
    \begin{pmatrix}
      A & B^* & 0 & \cdots\\
      B & A & B^* & \ddots \\
      0 & B  & \ddots & \ddots \\
      \vdots & \ddots  & \ddots & \ddots
    \end{pmatrix},
  \end{equation*}
  with entries $A$ and $B$ which are themselves Toeplitz on
  $\bigoplus_0^\infty \etK$.  Furthermore, $\deg_\pm A = d_A$, and
  $\deg_\pm B = d_B^\pm$.

  By Lemma~\ref{lem:Toeplitz-Ms}, there is are positive Toeplitz
  operators $\hat{A}$ and $\hat{M}$ on $\bigoplus_0^\infty \etK$,
  $A \geq \hat{A} \geq \hat{M}$, where $\hat{A}$ is a minimal positive
  Toeplitz operator such that
  \begin{equation*}
    \begin{pmatrix}
      \hat{A} - M & B^* \\ B & M
    \end{pmatrix},
  \end{equation*}
  for some $M \geq 0$ and $M = \hat{M}$ is the unique positive
  Toeplitz operator such that this holds for $\hat{A}$.  By
  Theorem~\ref{thm:min-A-w-degree-bds},
  $m:= \deg(\hat{M}) < \infty$, $m \geq d_B$.

  The entries of the operator
  $\begin{pmatrix} A - \hat{M} & B^* \\ B & \hat{M} \end{pmatrix} \geq
  0$ are Toeplitz of degree at most ${d'}_1 = \max\{d_A, m\}$.  These
  can be then be collected into $2 \times 2$ blocks giving a Toeplitz
  operator with entries acting on $\etK \oplus \etK$ of degree at most
  ${d'}_1$.  Applying the \fr theorem, there is a factorization in terms
  of analytic operators of at most degree ${d'}_1$, and hence
  \begin{equation*}
    L=
    \begin{pmatrix}
      A - \hat{M} & B^* \\ B & \hat{M}
    \end{pmatrix}
    =
    (F_{ij})^*_{i,j=1,2}(F_{ij})_{i,j=1,2},
  \end{equation*}
  where each $F_{i,j}$ is analytic of degree at most ${d'}_1$.

  Conjugate the terms of $L$ with the adjoint of the unitary operator
  $V$ from the first paragraph of the proof.  The operator $F_{ij}$
  becomes a $d_2 \times d_2$ operator $N_{ij}$ with entries that are
  analytic with degrees bounded by ${d'}_1$ on $\eG$.  Write
  $M = V \hat M V^*$.  Then
  \begin{equation*}
    \frac{1}{\sqrt{d_2}}
    \begin{pmatrix}
      \begin{pmatrix}
        R_0 & R_1^* & \cdots & R_{d_2-1}^* \\
        R_1 & \ddots & \ddots & \vdots \\
        \vdots & \ddots & \ddots & R_1^* \\
        R_{d_2-1} & \cdots & R_1 & R_0
      \end{pmatrix}
      - M &
      \begin{pmatrix}
        R_{d_2}^* & 0 & \cdots & 0  \\
        R_{d_2-1}^* & \ddots & \ddots & \vdots \\
        \vdots & \ddots & \ddots & 0 \\
        R_1^* & \ddots & R_{d_2-1}^* & R_{d_2}^*
      \end{pmatrix}
      \\
      \begin{pmatrix}
        R_{d_2} & R_{d_2-1} & \cdots & R_1  \\
        0 & \ddots & \ddots & \vdots \\
        \vdots & \ddots & \ddots & R_{d_2-1} \\
        0 & \ddots & 0 & R_{d_2}^*
      \end{pmatrix}
      & M
    \end{pmatrix}
    = (N_{ij})^*_{i,j = 1,2}(N_{ij})_{i,j = 1,2}.
  \end{equation*}
  For $m = 0, \dots, 2d_2-1$, define polynomials $\tilde F_m$ in $z_2$
  of degree at most $2d_2-1$ for which the coefficient of $z_2^k$ is
  the $k$th entry of the $m$th column of $(N_{ij})$.  Then
  \begin{equation*}
    \tilde Q := \sum_{-d_2}^{d_2} R_k z^k = \sum_0^{2d_2-1} \tilde
    F_k^*\tilde F_k.
  \end{equation*}
  Each $R_j$ corresponds to an analytic polynomial in $z_1$, as do all
  of the entries $N_{ijk\ell}$ of each $N_{ij}$ (here $k,\ell$ run
  from $1$ to $d_2$, and the polynomials have degree at most ${d'}_1$
  in $z_1$).  Replace the entries of $N_{ij}$ by the appropriate
  polynomials in $z_1$, and write $F_m$ for the resulting $2d_2$
  polynomials of degrees at most $({d'}_1,2d_2-1)$ in $(z_1,z_2)$.  It
  follows from the discussion towards the end of
  Section~\ref{sec:toepl-analyt-oper} that $Q= \sum_m F_m^* F_m$.
\end{proof}

Can anything be said when $\eH$ is not finite dimensional?  As noted
in Theorem~\ref{thm:mv-fr-thm}, if $Q$ is strictly positive, then it
has a factorization as a sum of hermitian squares of analytic
polynomials.  One could try to apply
Theorem~\ref{thm:min-A-gives-outer-fact} in the proof of
Theorem~\ref{thm:2d-fejer-riesz-thm} to an operator valued polynomial
which is positive but not strictly positive.  The result will be a
factorization as a sum of hermitian squares of analytic functions
which in general are not necessarily polynomials.

\section{Conclusion}
\label{sec:conclusion}

There is an aspect of the proof of
Theorem~\ref{thm:2d-fejer-riesz-thm} which is far from explicit, since
the construction of $\hat{M}$ there used a Zorn's lemma argument.
However on finite dimensional Hilbert spaces, by
Theorem~\ref{thm:min-A-w-degree-bds} the degree of $\hat{M}$ is
bounded, so it might nevertheless be the case that the construction of
this operator can be carried out concretely in many instances.
Undoubtedly, the coding will involve programming challenges.

There are numerous applications of
Theorem~\ref{thm:2d-fejer-riesz-thm}.  For example, via a Cayley
transform it is possible to construct rational factorizations of
positive matrix valued polynomials on $\mathbb R^2$~\cite{MAD2004}.
For strictly positive polynomials over $\mathbb R^n$, such a Cayley
transform gives factorizations involving a restricted class of
denominators, though it is known that for positive semidefinite
polynomials, this class may fail to be a finite~\cite{MR2159759} even
if the degrees are bounded.  The arguments from \cite{MAD2004} and the
result presented here imply that not only strictly positive, but also
non-negative matrix valued polynomials over $\mathbb R^2$, can be
factored using a restricted class of denominators, and that a finite
set of denominators works for all polynomials of bounded degree.

A number of papers have looked at the problem of factorization for
non-negative trigonometric polynomials in two variables, chiefly in
the context of engineering problems such as filter design.  These have
tended to restrict to polynomials having factorizations from the class
of stable polynomials; that is, polynomials with no zeros in the
closed bidisk~\cite{gw1,gw2,Knese1}, or to those with no zeros in the
closed disk crossed with the open disk~\cite{MR3273310}, or no zeros
on a face of the bidisk~\cite{MR3475462}.  Some address the scalar
case, while others the operator case.

As noted in the comment following the statement of
Theorem~\ref{thm:2d-fejer-riesz-thm}, there at least two
factorizations possible in the two variable setting.  Presumably this
is part of some larger family of factorizations.  Perhaps there is
some special ``central'' factorization, though it is unclear by how
much, if any, the bounds on the number of polynomials and their
degrees can be improved.  Once the minimal $\hat{A}$ is constructed,
the bi-infinite extension of the operator $\hat{M}$ is a Schur
complement, hence maximal in this context, and the ideas briefly
discussed in Section~\ref{sec:Schur-complements} for using Schur
complements to construct optimal \fr factorizations of one variable
trigonometric polynomials might improve the number and degree bounds
for the polynomials in the two variable factorization.

Finally, there are other related \emph{Positivstellens\"atze} in the
non-commutative setting which have not been touched upon (see, for
example, \cite{MR1815959,HMP1,MR2500470}).  What is observed though is
that in contrast to the commutative case, the less restrictive nature
of non-commutativity allows for factorization regardless of the number
of variables.



\end{document}